\documentclass[11pt]{amsart}
\usepackage{amsmath, amssymb, amscd, mathrsfs, url, pinlabel,hyperref, verbatim}
\usepackage[margin=1.25in,marginparwidth=1in,centering,letterpaper,dvips]{geometry}
\usepackage{color,dcpic,latexsym,graphicx,epstopdf,comment}
\usepackage[all]{xy}
\usepackage{tikz,pgfplots}
\usepackage{enumitem,upquote,tabularx,textcomp}
\usepackage[all]{hypcap}

\title{Sutured instanton homology and Heegaard diagrams}

\author{John A. Baldwin}
\address{Department of Mathematics\\Boston College}
\email{john.baldwin@bc.edu}

\author{Zhenkun Li}
\address{Department of Mathematics\\Stanford University}
\email{zhenkun@stanford.edu}

\author{Fan Ye}
\address{Department of Pure Mathematics and Mathematical Statistics\\University of Cambridge}
\email{fy260@cam.ac.uk}

\thanks{JAB was supported by NSF CAREER Grant DMS-1454865 and NSF FRG Grant DMS-1952707.}

\makeatletter
\newtheorem*{rep@theorem}{\rep@title}
\newcommand{\newreptheorem}[2]{%
\newenvironment{rep#1}[1]{%
 \def\rep@title{#2 \ref{##1}}%
 \begin{rep@theorem}}%
 {\end{rep@theorem}}}
\makeatother

\newtheorem {theorem}{Theorem}
\newreptheorem{theorem}{Theorem}
\newtheorem {lemma}[theorem]{Lemma}
\newtheorem {proposition}[theorem]{Proposition}
\newtheorem {corollary}[theorem]{Corollary}

\numberwithin{equation}{section}
\numberwithin{theorem}{section}

\theoremstyle{definition}

\newtheorem{remark}[theorem]{Remark}
\newtheorem*{remark*}{Remark}

\newlist{pcases}{enumerate}{1}
\setlist[pcases]{
  label=\bf{Case~\arabic*:}\protect\thiscase.~,
  ref=\arabic*,
  align=left,
  labelsep=0pt,
  leftmargin=0pt,
  labelwidth=0pt,
  parsep=0pt
}
\newcommand{\case}[1][]{%
  \if\relax\detokenize{#1}\relax
    \def\thiscase{}%
  \else
    \def\thiscase{~#1}%
  \fi
  \item
}

\newcommand{\Z}{\mathbb{Z}}
\newcommand{\rk}{\textrm{rk}}

\newcommand{\C}{\mathbb{C}}

\newcommand{\Q}{\mathbb{Q}}
\newcommand{\Cone}{\textrm{Cone}}

\newcommand{\T}{\mathbb{T}}
\newcommand{\gen}{\mathfrak{S}}

\DeclareMathOperator{\Sym}{Sym}

\newcommand\SFH{\mathit{SFH}}
\newcommand\SFC{\mathit{SFC}}
\newcommand\SHM{\mathit{SHM}}

\newcommand\SHI{\mathit{SHI}}

\newcommand\KHI{\mathit{KHI}}
\newcommand\HF{\widehat{\mathit{HF}}}
\newcommand\HFK{\widehat{\mathit{HFK}}}

\DeclareFontFamily{U}{mathx}{\hyphenchar\font45}
\DeclareFontShape{U}{mathx}{m}{n}{
      <5> <6> <7> <8> <9> <10>
      <10.95> <12> <14.4> <17.28> <20.74> <24.88>
      mathx10
      }{}
\DeclareSymbolFont{mathx}{U}{mathx}{m}{n}
\DeclareFontSubstitution{U}{mathx}{m}{n}
\DeclareMathAccent{\widecheck}{0}{mathx}{"71}

\newcommand{\Heeg}{\mathcal{H}}
\newcommand{\Met}{\mathcal{T}}
\newcommand{\D}{\mathbb{D}}

\usetikzlibrary{calc,intersections}
\tikzset{every picture/.style=thick}
\tikzset{link/.style = { white, double = black, line width = 1.75pt, double distance = 1.25pt, looseness=1.75 }}
\tikzset{crossing/.style = {draw, circle, dotted, minimum size=0.5cm, inner sep=0, outer sep=0}}
\pgfplotsset{compat=1.12}

\begin{document}

\begin{abstract}
Suppose $\Heeg$ is an admissible Heegaard diagram for a balanced sutured manifold $(M,\gamma)$. We prove that the number of generators of the associated sutured Heegaard Floer complex  is an upper bound on the dimension of the sutured instanton homology $\SHI(M,\gamma)$. It follows, in particular, that strong L-spaces are instanton L-spaces.

\end{abstract}

\maketitle

\section{Introduction} \label{sec:intro}

Let $(M,\gamma)$ be a balanced sutured manifold. Kronheimer and Mrowka conjectured   \cite{km-excision} that its sutured instanton homology  is isomorphic to its sutured Heegaard Floer homology, \begin{equation}\label{eqn:iso1}\SHI(M,\gamma)\cong \SFH(M,\gamma)\otimes \C.\end{equation} Proving this remains a major open problem. In particular, it would imply  isomorphisms  \begin{align*}\label{eqn:iso2}I^\#(Y)&\cong \HF(Y)\otimes \C\\
\KHI(Y,K)&\cong \HFK(Y,K)\otimes \C\end{align*} between the invariants of  closed 3-manifolds and knots  in the instanton and Heegaard Floer settings.

There has  been a  flood of recent work    proving these isomorphisms  for various     families of  closed 3-manifolds and knots; see  \cite{BSconcordance,abds,lpcs,li-ye,ghosh-li}. 
In this paper, we initiate a   systematic approach to  the general isomorphism \eqref{eqn:iso1}.
Before stating our main result, let us establish some notation. 

Given a  sutured Heegaard diagram  \[\Heeg = \big(\Sigma,\alpha = \{\alpha_1,\dots,\alpha_k\},\beta = \{\beta_1,\dots,\beta_k\}\big)\] for a balanced sutured manifold $(M,\gamma)$, let \[\T_\alpha := \alpha_1\times\dots\times\alpha_k \textrm{ and }\T_\beta:=\beta_1\times\dots\times\beta_k \subset \Sym^k(\Sigma)\] denote the usual tori in the $k$-fold symmetric product of $\Sigma$, and let  \[\gen(\Heeg) :=\T_\alpha\cap\T_\beta \subset \Sym^k(\Sigma).\] If $\Heeg$ is \emph{admissible}, then $\gen(\Heeg)$ is the set of generators for the sutured Heegaard Floer complex $\SFC(\Heeg)$ as defined by Juh{\'a}sz in \cite{juhasz-sutured}.\footnote{See \S\ref{ssec:admissible} for the definition of \emph{admissible};  every $\Heeg$ is  admissible when   $H_1(M,\partial M;\Q)=0$.} Our main theorem is the following.

\begin{theorem}
\label{thm:main}
If $\Heeg$ is an admissible  sutured Heegaard diagram for  $(M,\gamma)$, then \[\dim_\C \SHI(M,\gamma)\leq |\gen(\Heeg)|.\]
\end{theorem}

\begin{remark}
Theorem \ref{thm:main} does not hold without the assumption that $\Heeg$ is admissible; see Remark \ref{rmk:admissible}. 
\end{remark}

\begin{remark}
Our proof of Theorem \ref{thm:main} also  works for sutured monopole homology ($\SHM$) in place of $\SHI$. Of course, the $\SHM$ version  of our main result follows from the  isomorphism  \[\SHM(M,\gamma) \cong \SFH(M,\gamma),\] which is a consequence of the equivalence between monopole and Heegaard Floer homology; see \cite{lekili2}. Still, it may   be of value to know that one can prove the inequality \[\dim_\Z \SHM(M,\gamma)\leq |\gen(\Heeg)|\] without going through the proof of this equivalence.
\end{remark}

Given a balanced sutured manifold $(M,\gamma)$, we define the \emph{simultaneous trajectory number} $\Met(M,\gamma)$ to be the minimum of $|\gen(\Heeg)|$ over all admissible sutured Heegaard diagrams $\Heeg$ for $(M,\gamma)$.
This is the generalization to balanced sutured manifolds of a notion originally defined  for rational homology 3-spheres by Ozsv{\'a}th and Szab{\'o} in \cite{osz-props}. It admits   a  purely Morse-theoretic interpretation when   $H_1(M,\partial M;\Q)=0$, and is a measure of the topological complexity of $(M,\gamma)$---for example,   $\Met(M,\gamma)=1$  iff $(M,\gamma)$ is a product sutured manifold.\footnote{This is a fun exercise which we have not seen written down before; it is a generalization of the well-known fact that $S^3$ is the only closed 3-manifold with simultaneous trajectory number one \cite{osz-props}.} Further, it is clear from the definition that \[\rk_\Z\SFH(M,\gamma)\leq \Met(M,\gamma).\] We have the following immediate corollary of Theorem \ref{thm:main}.

\begin{corollary}
If $(M,\gamma)$ is a balanced sutured manifold, then \[\dim_\C\SHI(M,\gamma)\leq \Met(M,\gamma).\]  \end{corollary}

For the  corollaries below, we recall the natural sutured manifolds associated to closed 3-manifolds and knots therein. Given a closed 3-manifold $Y$, let $(Y(1),\delta)$ denote the sutured manifold obtained by removing a 3-ball from $Y$, where $\delta$ is a simple closed curve on $\partial Y(1)\cong S^2$. Similarly, given a knot $K\subset Y$, let $(Y(K),m \cup -m)$ be the sutured manifold obtained by removing a tubular neighborhood of $K$, where $m$ and $-m$ are oppositely oriented meridional curves on $\partial Y(K)\cong T^2$. The framed instanton  and Heegaard Floer homologies of a closed 3-manifold $Y$ are given by \begin{align*}
I^\#(Y) &\cong \SHI(Y(1),\delta) \\
\HF(Y) &\cong \SFH(Y(1),\delta).
\end{align*} 
Likewise the instanton and Heegaard knot Floer homologies of a knot $K\subset Y$ are given by
\begin{align*}\KHI(Y,K) &\cong \SHI(Y(K),m\cup-m)\\\HFK(Y,K) &\cong\SFH(Y(K),m\cup-m).\end{align*}

Recall that the ranks of the  Heegaard Floer homology and framed instanton homology  of a rational homology 3-sphere $Y$ are each bounded below by $|H_1(Y)|$. An \emph{L-space}, respectively \emph{instanton L-space}, is a rational homology 3-sphere which achieves these  lower bounds
 \begin{align*}\rk_\Z \HF(Y) &= |H_1(Y)|, \\
 \dim_\C I^\#(Y) &= |H_1(Y)|,\end{align*} respectively.
 A  \emph{strong L-space}, as defined by Levine and Lewallen in  \cite{levine-lewallen}, is a rational homology 3-sphere $Y$ which satisfies the stronger condition \[ \Met(Y(1),\delta) = |H_1(Y)|.\footnote{Strong L-spaces are also of interest  because Levine and Lewallen were able to show that their fundamental groups are not left-orderable, as predicted by the L-space Conjecture.}  \] Indeed, this condition implies  that $Y$ is an L-space, since  \[|H_1(Y)|
\leq \rk_\Z \HF(Y) \leq \Met(Y(1),\delta).\]
The following is then an immediate corollary of Theorem \ref{thm:main}.

\begin{corollary}
If $Y$ is a strong L-space, then it is an instanton L-space. 
\end{corollary}

More generally,  we say that a sutured Heegaard diagram $\Heeg$ for a  sutured manifold $(M,\gamma)$ is \emph{strong} if $\Heeg$ is admissible and the   sutured Floer complex $\SFC(\Heeg)$ has trivial differential. We then have the following.

\begin{corollary}
If $(M,\gamma)$ is a balanced sutured manifold which has a strong sutured Heegaard diagram, then \[\dim_\C\SHI(M,\gamma)\leq \rk_\Z\SFH(M,\gamma).\] 
\end{corollary}

For example, when $K$ is a $(1,1)$-knot in a lens space $L(p,q)$,  the  sutured manifold \[\SHI((L(p,q))(K),m\cup-m)\] has a strong  Heegaard diagram. We thus reproduce the following  result of Li and Ye \cite{li-ye}.

\begin{corollary}
\label{cor:simple}
If $K\subset L(p,q)$ is a $(1,1)$-knot, then \[\dim_\C\KHI(L(p,q),K)\leq \rk_\Z\HFK(L(p,q),K).\] \end{corollary}




\subsection{On the proof} Given a vertical tangle $T$ in a balanced sutured manifold $(M,\gamma)$, one forms an associated  sutured manifold $(M_T,\gamma_T)$ by removing a neighborhood of $T$ from $M$, and adding meridians  of the components of $T$ to $\gamma$; see \S\ref{sec:inequality} for more details. Li and Ye proved the following dimension inequality  in \cite[Proposition 3.14]{li-ye}.

\begin{theorem}
\label{thm:tangleinequality}
If $T$ is a vertical tangle in $(M,\gamma)$ such that $[T_i]=0$ in $H_1(M,\partial M;\Q)$ for each component $T_i$ of $T$, then \[\dim_\C \SHI(M,\gamma)\leq \dim_\C\SHI(M_T,\gamma_T).\]\end{theorem}

To prove Theorem \ref{thm:main}, we  first establish the same inequality under the weaker assumption that  $T$ (rather than each of its components) is rationally nullhomologous, in \S\ref{sec:inequality}.

\begin{theorem}
\label{thm:inequality2} 
If $T$ is a vertical tangle in $(M,\gamma)$ such that $[T]=0$ in $H_1(M,\partial M;\Q)$, then \[\dim_\C\SHI(M,\gamma)\leq \dim_\C\SHI(M_T,\gamma_T).\]
\end{theorem}

Next, given an admissible sutured Heegaard diagram $\Heeg$ for $(M,\gamma)$, we construct  a  vertical tangle $T\subset (M,\gamma)$ with $[T]=0$ in $H_1(M,\partial M;\Q)$, such that \[\dim_\C\SHI(M_T,\gamma_T) = |\gen(\Heeg)|.\] This is the content of \S\ref{sec:proof}. Theorem \ref{thm:main} then follows from Theorem \ref{thm:inequality2}. 

\begin{remark}Theorem \ref{thm:tangleinequality} suffices to prove Theorem \ref{thm:main} in the case where  $H_1(M,\partial M;\Q)=0$, but we need the stronger Theorem \ref{thm:inequality2} in general.
\end{remark}

\subsection{Organization} In \S\ref{sec:inequality}, we prove the  inequality in Theorem \ref{thm:inequality2}. We then use this in \S\ref{sec:proof} to prove our main result, Theorem \ref{thm:main}. Finally, in \S\ref{sec:further}, we discuss further directions, some of which are in progress. In particular, we discuss the possibility of using the proof of Theorem \ref{thm:main} to construct a grading on $\SHI$ by homotopy classes of 2-plane fields, and the prospects for upgrading Theorem \ref{thm:main} to a proof of the isomorphism \eqref{eqn:iso1}.

\subsection{Acknowledgements} We thank Adam Levine, Tye Lidman,  and Steven Sivek for  helpful conversations. 

\section{A dimension inequality}\label{sec:inequality}

A \emph{vertical tangle} $T=T_1\cup \dots \cup T_n$ in a balanced sutured manifold $(M,\gamma)$ is a properly embedded 1-manifold in $M$, with boundary in $R(\gamma)$,   whose components $T_i$ are oriented from $R_+(\gamma)$ to $R_-(\gamma)$.
One forms an associated  balanced sutured manifold $(M_T,\gamma_T)$ by removing tubular neighborhoods of the components $T_i$, and adding  positively-oriented meridians of these components to the suture $\gamma$, as in \cite[Section 3]{li-ye} and depicted in Figure \ref{fig:tangle}.  In this section, we prove Theorem \ref{thm:inequality2}, which states that \[\dim_\C \SHI(M,\gamma)\leq \dim_\C\SHI(M_T,\gamma_T)\] when $[T]=0$ in $H_1(M,\partial M;\Q)$. The  rough idea  is to turn $T$ into a related tangle $T'$ whose components are rationally nullhomologous, and apply Theorem \ref{thm:tangleinequality}.

\begin{proof}[Proof of Theorem \ref{thm:inequality2}] Let $T_1,\dots,T_n$ be the components of $T$. For $i=1,\dots,n$, let \[\partial T_i = q_i-p_i,\] where $p_i\in R_+(\gamma)$ and $q_i \in R_-(\gamma)$. We may assume that $\gamma$ is connected, since we can achieve this by adding contact 1-handles to $(M,\gamma)$, an operation which does not change $\SHI(M,\gamma)$ (equivalently, $\SHI$ is invariant under product disk decomposition \cite{km-excision,bs-instanton}). Then we can find a sequence of pairwise disjoint  arcs \[\xi_1,\dots,\xi_n\subset\partial M\] such that, for each $i=1,\dots,n$, we have:
\begin{itemize}
\item $\partial\xi_i = p_{i+1}-q_i$ (where $p_{n+1}:=p_1$), and
\item $\xi_i$ intersects $\gamma$ in exactly one point.
\end{itemize} For every $i\in\{2,\dots,n\}$,  choose an arc $t_{i} \subset \partial M$ in a neighborhood of the unique intersection point $\xi_{i-1}\cap \gamma$, as depicted in Figures \ref{fig:ti} and \ref{fig:arcs}. Push the interior of $t_{i}$ into the interior of $M$ to turn this arc into a  vertical tangle $T_{i}'$, and  let  \[T'' = T_2' \cup\dots\cup T_n'.\] Then \[(M_{T''},\gamma_{T''}) = (M-N(T''),\gamma\cup \mu_2'\cup\dots\cup\mu_n'), \] where $\mu_i'$ is a positively-oriented meridian of $T_i'$.  Each component $T_{i}'$ cobounds a disk in $M$ with the arc $t_{i}$. These disks then restrict  to properly embedded disks \[D_2,\dots, D_n\subset M_{T''}\] with $|D_i\cap \gamma|=1$ and $|D_i\cap \mu_j'|=\delta_{ij}$, so that \[|D_i\cap \gamma_{T''}|=2.\] Thus, each $D_i$ is a product disk. 

\begin{figure}[ht]
\labellist
\tiny \hair 2pt
\pinlabel $t_i$ at 73 3
\pinlabel $\xi_{i-1}$ at 136 48
\pinlabel $R_-(\gamma)$ at 12 65
\pinlabel $R_+(\gamma)$ at 50 65
\pinlabel $\partial M$ at 150 70
\pinlabel $\gamma$ at 129 28
\endlabellist
\centering
\includegraphics[width=3.5cm]{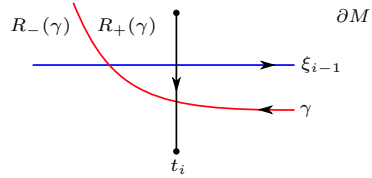}
\caption{The point of view is from the interior of $M$, looking at $\partial M$.}
\label{fig:ti}
\end{figure}

\begin{figure}[ht]
\labellist
\pinlabel $\dots$ at 200 59
\tiny \hair 2pt
\pinlabel $M$ at 272 59
\pinlabel $T_1$ at 48 70
\pinlabel $T_2$ at 98 70
\pinlabel $T_3$ at 150.5 70
\pinlabel $T_n$ at 238.5 70
\pinlabel $\xi_1$ at 68 30
\pinlabel $\xi_3$ at 167 30
\pinlabel $\xi_2$ at 120 94
\pinlabel $\xi_n$ at 215 108
\pinlabel $p_1$ at 48 98
\pinlabel $q_1$ at 48 24
\pinlabel $p_n$ at 251 24
\pinlabel $q_n$ at 251 98
\pinlabel $t_2$ at 87 35.5
\pinlabel $t_3$ at 149 89
\pinlabel $t_{n}$ at 229 35.5
\endlabellist
\centering
\includegraphics[width=10cm]{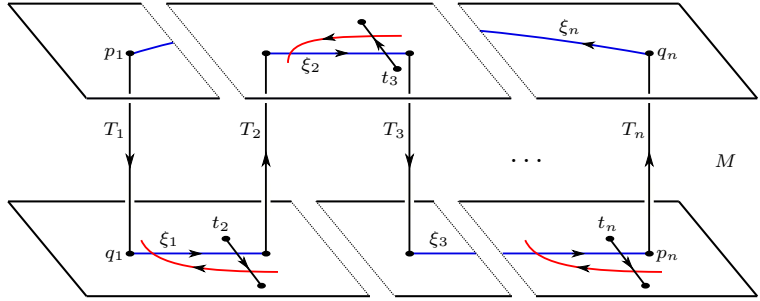}
\caption{The tangle $T = T_1\cup \dots \cup T_n$ in $M$ and the arcs $\xi_i$ and $t_i$ in $\partial M$. The suture $\gamma$ is shown in red.}
\label{fig:arcs}
\end{figure}

Next, consider the arc \[t_1 = T_1\cup\xi_1\cup T_2\cup\xi_2\cup\dots\cup T_{n-1}\cup \xi_{n-1}\cup T_n \subset M_{T''}.\] Push its interior into the interior of $M_{T''}$ to form a vertical tangle $T_1'$ with $\partial T_1' = q_n-p_1,$ as in Figure \ref{fig:T}.  Let $T'$ be the tangle in $M$ given by \[T' = T_1'\cup T'' = T_1'\cup \dots\cup T_n'.\] We will refer to a tangle $T'$ formed in this way as a \emph{mixed tangle for $T$}. Note that \[(M_{T'},\gamma_{T'}) = ((M_{T''})_{T_1'}, (\gamma_{T''})_{T_1'}) = (M-N(T'),\gamma\cup \mu_1'\cup\dots\cup\mu_n'), \] where $\mu_1'$ is a positively-oriented meridian of $T_1'$. Observe that the  disks $D_2,\dots, D_n\subset M_{T''}$ restrict to properly embedded annuli \[A_2,\dots,A_n\subset M_{T'}\] since $T_1'$ intersects each disk  in exactly one point, as shown in Figure \ref{fig:annulus}. 

\begin{figure}[ht]
\labellist
\pinlabel $\dots$ at 200 59
\tiny \hair 2pt
\pinlabel $M$ at 272 59
\pinlabel $T_1'$ at 48 70

\pinlabel $\xi_n$ at 215 108
\pinlabel $p_1$ at 48 98
\pinlabel $q_n$ at 251 98
\pinlabel $T_2'$ at 87 38
\pinlabel $T_3'$ at 136 90
\pinlabel $T_n'$ at 228 38
\endlabellist
\centering
\includegraphics[width=10cm]{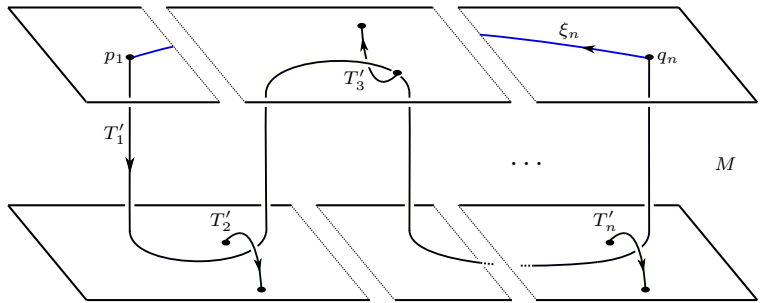}
\caption{The tangle $T' = T_1' \cup \dots \cup T_n'$ in $M$.}
\label{fig:T}
\end{figure}

\begin{figure}[ht]
\labellist
\tiny \hair 2pt
\pinlabel $M_{T'}$ at 210 105
\pinlabel $\partial N(T_1')$ at 9 53
\pinlabel $\partial N(T_i')$ at 147 30

\pinlabel $\mu_i'$ at 126 116
\pinlabel $\mu_1'$ at 52 51
\pinlabel $\gamma$ at 187 83
\pinlabel $A_i$ at 94 92
\pinlabel $R_+$ at 43 82
\pinlabel $R_-$ at 106 8
\pinlabel $R_+$ at 26 37.5
\pinlabel $R_-$ at 71 57
\endlabellist
\centering
\includegraphics[width=6cm]{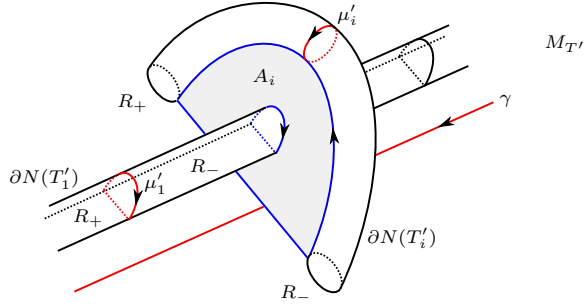}
\caption{The annulus $A_i$ in $M_{T'}$ near the boundaries of the tubular neighborhoods of the components $T_1'$ and $T_i'$, as seen from inside $M_{T'}$.}
\label{fig:annulus}
\end{figure}

The endpoints of the arc $\xi_n\subset \partial M_{T''}$ agree with $\partial T_1'$, and $|\xi_n\cap\gamma|=1$.  We can thus use $\xi_n$ together with  $\gamma$ and $\mu_1'$ to define a sequence of  sutures $\Gamma_m\subset \partial M_{T'}$ for $m\in\mathbb{N}$, as in \cite[Section 3.2]{li-ye}, which one should regard as ``longitudinal" sutures for $T_1'$; see Figure \ref{fig:annulus2}. By the construction of $T_1'$ and the assumption  that $T$ is rationally nullhomologous in $(M,\partial M)$, we have  \[[T_1'] = [T] =0 \in H_1(M_{T''},\partial M_{T''};\Q).\] Therefore, by \cite[Lemmas 3.21 and 3.22]{li-ye}, we have  the following.

\begin{lemma}
\label{lem:tri1} There is an exact triangle  \[ \xymatrix@C=-40pt@R=30pt{
\SHI(-M_{T'},-\Gamma_m) \ar[rr] & & \SHI(-M_{T'},-\Gamma_{m+1}) \ar[dl]^{F_{m+1}} \\
& \SHI(-M_{T''},-\gamma_{T''}), \ar[ul]^{G_m} & \\
} \] coming from the surgery exact triangle associated to surgeries on the meridian $\mu_1'$ of $T_1'$. Furthermore, $G_m\equiv 0$ for $m$ sufficiently large.
\end{lemma}

\begin{figure}[ht]
\labellist
\tiny \hair 2pt
\pinlabel $-M_{T'}$ at 383 105
\pinlabel $-\Gamma_{m+1}$ at 27 60
\pinlabel $-\gamma_{T'}$ at 210 60
\pinlabel $A_i$ at 94 92
\pinlabel $A_i^-$ at 272 92
\endlabellist
\centering
\includegraphics[width=12.5cm]{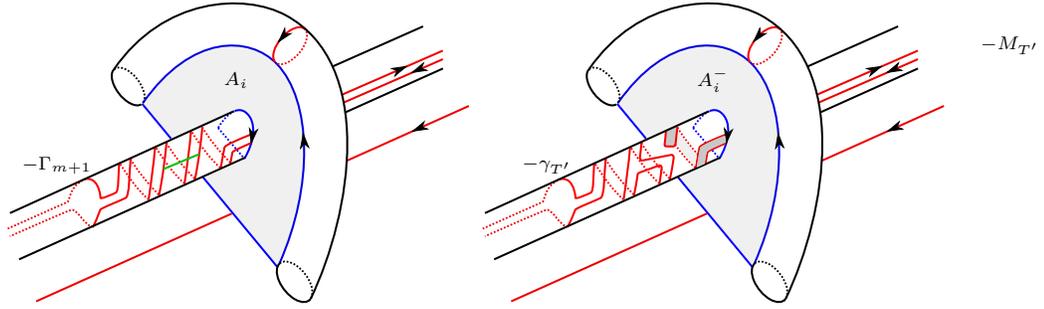}
\caption{Left, the suture $-\Gamma_{m+1}$. The bypass arc $\eta_-$ is shown in green. Right, the suture $-\gamma_{T'}$ resulting from the bypass attachment along $\eta_-$, and the negatively stabilized annulus $A_i^-\subset (-M_{T'},-\gamma_{T'})$.}
\label{fig:annulus2}
\end{figure}

For each $i=2,\dots,n$ and every $m\in \mathbb{N}$, we have that \[|A_i\cap \Gamma_m|=4.\] Let us orient each $A_i$ so that the induced orientation on $\partial A_i$ is opposite  the orientation of $\partial D_i$ coming from that of $T_i'$, as in Figures \ref{fig:annulus} and \ref{fig:annulus2}. By \cite{ghosh-li}, the disks $D_2,\dots,D_n$ induce a $\Z^{n-1}$-grading on \[\SHI(-M_{T''},-\gamma_{T''}).\] Similarly, the annuli $A_2,\dots,A_n$ induce a $\Z^{n-1}$-grading on \[\SHI(-M_{T'},-\Gamma_m)\] for each $m\in \mathbb{N}$, and we have the following graded version of the triangle in Lemma \ref{lem:tri2}.

\begin{lemma}
\label{lem:tri2} The exact triangle of Lemma \ref{lem:tri1} restricts to the exact triangle  \[ \xymatrix@C=-96pt@R=35pt{
\SHI(-M_{T'},-\Gamma_m, (A_2,\dots,A_n), (0,\dots,0)) \ar[rr] & & \SHI(-M_{T'},-\Gamma_{m+1},(A_2,\dots,A_n), (0,\dots,0)) \ar[dl]^{F_{m+1}} \\
& \SHI(-M_{T''},-\gamma_{T''},(D_2,\dots,D_n), (0,\dots,0)). \ar[ul]^{G_m} & \\
} \] 
\end{lemma}

\begin{remark}
A more general version of Lemma \ref{lem:tri2} will be proved in \cite{li-ye2}.
\end{remark}

\begin{proof}[Proof of Lemma \ref{lem:tri2}] We will  prove that  the map $F_{m+1}$ preserves the gradings. The arguments for the other two maps are similar.

Let us first recall the definition of $F_{m+1}$ from \cite[Section 3]{li-ye}. Pick a \emph{closure} \[(Y_{m+1},R_{m+1},\omega_{m+1})\] for $(-M_{T'},-\Gamma_{m+1})$ so that each annulus $A_i$ extends to a closed surface $\bar A_i\subset Y_{m+1}$, as in \cite[Section 3]{li}. By the construction therein, \[g(\bar A_i)=2\] for  $i=2,\dots,n$, since each component of $\partial A_i$ intersects $\Gamma_{m+1}$ in two points. The sutured instanton homology of $(-M_{T'},-\Gamma_{m+1})$ is  defined as a certain direct summand \[\SHI(-M_{T'},-\Gamma_{m+1}) = I_*(Y_{m+1}|R_{m+1})_{\omega_{m+1}}\] of the instanton Floer homology $I_*(Y_{m+1})_{\omega_{m+1}}$ as in \cite[Section 7]{km-excision}. The summand \[\SHI(-M_{T'},-\Gamma_{m+1},(A_2,\dots,A_n), (i_2,\dots,i_n))\] is defined as the simultaneous (generalized) $(2i_2,\dots,2i_n)$-eigenspace of commuting operators \[\mu(\bar A_2),\dots,\mu(\bar A_n):\SHI(-M_{T'},-\Gamma_{m+1}) \to \SHI(-M_{T'},-\Gamma_{m+1})\] associated to these surfaces.

The meridian $\mu_1'$ of $T_1'$ can be thought of as an embedded circle in $Y_{m+1}$. Let $Y$ be the manifold obtained from $Y_{m+1}$ via $0$-surgery on $\mu_1'$, with respect to the framing of $\mu_1'$ induced by $\partial M_{T'}$.  Since $\mu_1'$ is disjoint from $R_{m+1}$ and  the $\bar A_i$, these surfaces survive in $Y$.
By \cite[Section 3.3]{bs-instanton}, $(Y,R_{m+1},\omega_{m+1})$ is a closure of $(-M_{T''}, -\gamma_{T''})$. The map $F_{m+1}$ is then induced by the cobordism given by the trace of $0$-surgery on $\mu_1'$. Since $\bar A_i\subset Y_{m+1}$ is homologous to $\bar A_i\subset Y$ in this  cobordism,  $F_{m+1}$ respects the eigenspaces of  $\mu(\bar A_i)$. Thus, $F_{m+1}$ maps  \[\SHI(-M_{T'},-\Gamma_{m+1},(A_2,\dots,A_n), (i_2,\dots,i_n))\] into \[\SHI(-M_{T''},-\gamma_{T''},(A_2,\dots,A_n), (i_2,\dots,i_n)).\]

Now, the $0$-surgery on $\mu_1'$ makes each $\bar A_i$ compressible in $Y$; in particular,  each $\bar A_i\subset Y$ is homologous to the disjoint union of two tori \[T_i^1\cup T_i^2\subset Y.\] One of these tori, say $T_i^1$, is the extension $\bar D_i\subset Y$ of $D_i\subset -M_{T''}$ that is used to define the grading on $\SHI(-M_{T''},-\gamma_{T''})$ associated to $D_i$. Since \[\bar A_i=\bar D_i+T_i^2\] in $H_2(Y)$, the $k$-eigenspace of $\mu(\bar A_i)$ agrees with the $k$-eigenspace of $\mu(\bar D_i)$ for every $k$, by  \cite[Corollary 2.9]{bs-trefoil}. Thus, we have that \begin{align*}&\SHI(-M_{T''},-\gamma_{T''},(A_2,\dots,A_n), (i_2,\dots,i_n)) \\
=\,&\SHI(-M_{T''},-\gamma_{T''},(D_2,\dots,D_n), (i_2,\dots,i_n)).\end{align*} Putting these arguments together, we see that $F_{m+1}$ preserves the $\Z^{n-1}$-gradings as claimed in the lemma.
\end{proof}

Note that decomposing $(-M_{T''},-\gamma_{T''})$ along  $D_2\cup \dots \cup D_n$ yields  $(-M,-\gamma)$. By \cite[Lemma 4.2]{li}, we therefore have  \[\SHI(-M_{T''},-\gamma_{T''},(D_2,\dots,D_n), (0,\dots,0))\cong \SHI(-M,-\gamma).\] Hence, for $m$ sufficiently large, Lemmas \ref{lem:tri1} and \ref{lem:tri2} imply that that \begin{align} 
\label{eqn:2.1}\dim_\C\SHI(-M,-\gamma) =&\dim_\C\SHI(-M_{T'},-\Gamma_{m+1},(A_2,\dots,A_n), (0,\dots,0))\\
\nonumber&-\dim_\C\SHI(-M_{T'},-\Gamma_{m},(A_2,\dots,A_n), (0,\dots,0)),
\end{align}
since $G_m\equiv 0$.

Next, we consider attaching a bypass to $(-M_{T'},-\Gamma_{m+1})$ along the arc $\eta_-$  in Figure \ref{fig:annulus2}. By \cite[Section 4]{bs-trefoil}, this attachment gives rise to a bypass exact triangle. As discussed in \cite[Section 3]{li-ye}, the other two sutures involved in the triangle are $-\Gamma_m$ and $ -\gamma_{T'}$. It is straightforward to check that the bypass attachment along $\eta_-$ creates a \emph{negative stabilization} \[A_i^-\subset (-M_{T'},-\gamma_{T'})\] of $A_i$, for each $i=2,\dots,n$, in the sense of \cite[Definition 3.1]{li}. Hence, as in the proof of \cite[Proposition 5.5]{li}, we have  the following graded version of the bypass exact triangle of \cite[Theorem 1.20]{bs-trefoil},
\begin{equation*}\label{eqn:tri} \xymatrix@C=-96pt@R=35pt{
\SHI(-M_{T'},-\Gamma_m, (A_2,\dots,A_n), (0,\dots,0)) \ar[rr] & & \SHI(-M_{T'},-\Gamma_{m+1},(A_2,\dots,A_n), (0,\dots,0)) \ar[dl] \\
& \SHI(-M_{T'},-\gamma_{T'},(A_2^-,\dots,A_n^-), (0,\dots,0)), \ar[ul] & \\
} \end{equation*} which implies that 
\begin{align} 
\label{eqn:2.2}&\dim_\C\SHI(-M_{T'},-\gamma_{T'},(A_2^-,\dots,A_n^-), (0,\dots,0))\\
 \nonumber\geq&\dim_\C\SHI(-M_{T'},-\Gamma_{m+1},(A_2,\dots,A_n), (0,\dots,0))\\
\nonumber&-\dim_\C\SHI(-M_{T'},-\Gamma_{m},(A_2,\dots,A_n), (0,\dots,0)).
\end{align}
From the grading shifting property \cite[Theorem 1.12]{li} and \cite[Proposition 4.1]{wang}, we have \begin{align}
\label{eqn:2.3}&\SHI(-M_{T'},-\gamma_{T'},(A_2^-,\dots,A_n^-), (0,\dots,0)) \\
\nonumber =\,\, &\SHI(-M_{T'},-\gamma_{T'},(A_2^+,\dots,A_n^+), (-1,\dots,-1)),
\end{align} where $A_i^+$ is a \emph{positive stabilization} of $A_i$. Moreover, from the construction of the gradings and stabilizations in \cite[Section 3]{li}, we have 
\begin{align}
\label{eqn:2.4}&\SHI(-M_{T'},-\gamma_{T'},(A_2^+,\dots,A_n^+), (-1,\dots,-1)) \\
\nonumber =\,\, &\SHI(-M_{T'},-\gamma_{T'},(-(A_2^+),\dots,-(A_n^+)), (1,\dots,1))\\
\nonumber =\,\, &\SHI(-M_{T'},-\gamma_{T'},((-A_2)^-,\dots,(-A_n)^-), (1,\dots,1)).
\end{align} By \cite[Lemma 4.2]{li}, this last group is isomorphic to the sutured instanton homology of the manifold obtained from $(-M_{T'},-\gamma_{T'})$ by decomposing along $(-A_2)^-\cup \dots\cup(-A_n)^-$. By \cite[Lemma 3.2]{li}, this is the same as the manifold obtained by decomposing along $-A_2\cup \dots\cup-A_n$, which is, after reversing orientation,  the manifold obtained from $(M_{T'},\gamma_{T'})$ by decomposing along $A_2\cup \dots\cup A_n$. It is straightforward to check that the latter manifold  is simply $(M_T,\gamma_T)$, as indicated in Figure \ref{fig:annulus3} in the case $n=2$. Thus,\begin{equation}
\label{eqn:2.5}\SHI(-M_{T'},-\gamma_{T'},((-A_2)^-,\dots,(-A_n)^-), (1,\dots,1)) \cong \SHI(-M_T,-\gamma_T).
\end{equation}

\begin{figure}[ht]
\labellist
\tiny \hair 2pt
\pinlabel $M_{T}$ at 210 105

\pinlabel $\mu_1$ at 52 51
\pinlabel $\mu_2$ at 112 77

\pinlabel $\gamma$ at 187 83
\pinlabel $A_2$ at 94 92
\endlabellist
\centering
\includegraphics[width=6cm]{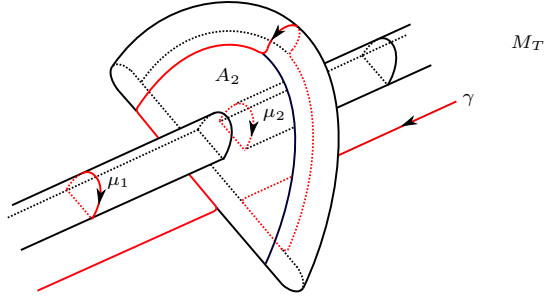}
\caption{The result of decomposing $(M_{T'},\gamma_{T'})$ along $A_2\cup \dots \cup A_n$ is simply $(M_T,\gamma_T)$. This is illustrated above in the case $n=2$.}
\label{fig:annulus3}
\end{figure}

Finally, combining \eqref{eqn:2.1}--\eqref{eqn:2.5}, we have that
\begin{align*}
\dim_\C\SHI(-M_T,-\gamma_T)=&\dim_\C\SHI(-M_{T'},-\gamma_{T'},((-A_2)^-,\dots,(-A_n)^-), (1,\dots,1))\\
=&\dim_\C\SHI(-M_{T'},-\gamma_{T'},(A_2^-,\dots,A_n^-), (0,\dots,0)) \\
\geq&\dim_\C\SHI(-M_{T'},-\Gamma_{m+1},(A_2,\dots,A_n), (0,\dots,0))\\
&-\dim_\C\SHI(-M_{T'},-\Gamma_{m},(A_2,\dots,A_n), (0,\dots,0))\\
=&\dim_\C\SHI(-M,-\gamma).
\end{align*}
Given the symmetry of $\SHI$ under orientation reversal, this proves Theorem \ref{thm:inequality2}.
\end{proof}

\section{Proof of Theorem \ref{thm:main}}\label{sec:proof} 

\subsection{Full tangles}\label{ssec:full}
 Let $(M,\gamma)$ be a balanced sutured manifold. Let \[\Heeg = \big(\Sigma,\alpha = \{\alpha_1,\dots,\alpha_k\},\beta = \{\beta_1,\dots,\beta_k\}\big)\] be any (not necessarily admissible) sutured Heegaard diagram for $(M,\gamma)$. This means that  $M$ is  obtained from $\Sigma\times[-1,1]$ by attaching  3-dimensional 2-handles \begin{align*}
 \D_{\alpha_i} &=D_{\alpha_i}^2\times I\\
 \D_{\beta_i} &=D_{\beta_i}^2\times I
 \end{align*}
 along  $A_i\times\{-1\}$ and $B_i\times\{+1\}$, where $A_i$ and $B_i$ are annular neighborhoods of $\alpha_i$ and $\beta_i$, respectively, for $i=1,\dots,k$. The suture $\gamma$  is given by \[\gamma = \partial \Sigma\times \{0\}.\footnote{This construction describes $(M,\gamma)$ up to homeomorphism, which is all we are concerned with.}\] We will next define a special class of vertical tangles in $(M,\gamma)$ associated to $\Heeg$.
 
Let $R_1,\dots,R_n$ be the regions of $\Sigma-\alpha-\beta$ disjoint from $\partial \Sigma$. For each $i=1,\dots,n$,  let $p_{i1},\dots,p_{ia_i}$ be $a_i$ distinct points in  $R_i$, for some integer $a_i\geq 1$. Let \[T_{ij} = p_{ij}\times[-1,1]\subset \Sigma\times[-1,1],\] and let \[T =  \bigcup_{i=1}^n \bigcup_{j=1}^{a_i} T_{ij}.\] Then $T$ is a vertical tangle in $(M,\gamma)$, oriented from $R_+(\gamma)$ to $R_-(\gamma)$. Let $D_{ij}^2$ be a tubular neighborhood of the point $p_{ij}\in R_i$,  let \[N_{ij} =D^2_{ij}\times [-1,1]\] be a tubular neighborhood of the component $T_{ij}$ in $M$, and let \[\gamma_{ij} = \partial D^2_{ij}\times \{0\}\subset \partial N_{ij}\] be a positively-oriented meridian of this component; see Figure \ref{fig:tangle}. Let $(M_T, \gamma_T)$ be the  balanced sutured manifold obtained from $M$ by removing these tubular neighborhoods, \[M_T = M-\bigcup_{i=1}^n\bigcup_{j=1}^{a_i}N_{ij},\] where $\gamma_T$ is the union of $\gamma$ with meridians of the $T_{ij}$, \[\gamma_T = \gamma\cup \bigcup_{i=1}^n\bigcup_{j=1}^{a_i}\gamma_{ij},\] as in \S\ref{sec:inequality}. 
We  refer to any tangle obtained in this way as a \emph{full tangle for $\Heeg$}.

\begin{figure}[ht]
\labellist
\tiny \hair 2pt
\pinlabel $1$ at 10 85
\pinlabel $-1$ at 5 44
\pinlabel $T_{ij}$ at 136 74
\pinlabel $\gamma_{ij}$ at 348 74
\endlabellist
\centering
\includegraphics[width=11cm]{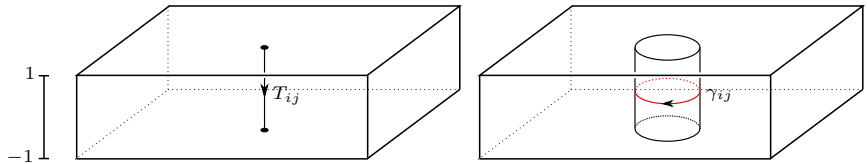}
\caption{Left, the component $T_{ij}=p_{ij}\times[-1,1]\subset M$. Right, the complement of $N_{ij}$ with the meridian $\gamma_{ij}$ in red.}
\label{fig:tangle}
\end{figure}

The main result of this section is the following.

\begin{proposition}
\label{prop:gens}
If $T$ is a full tangle for $\Heeg$, then $\dim_\C\SHI(M_T,\gamma_T) = |\gen(\Heeg)|.$
\end{proposition}

\begin{remark}
\label{rmk:SFH}
The analogue of this proposition for $\SFH$ is immediate. One forms a sutured Heegaard diagram $\Heeg_T$ for $(M_T,\gamma_T)$ from $\Heeg$ by removing neighborhoods  of the  $p_{ij}\in \Sigma$. Since there is at least one such point in every region of $\Sigma-\alpha-\beta$ not intersecting $\partial\Sigma$,  this ensures that $\Heeg_T$ is admissible and that the differential on $\SFC(\Heeg_T)$ is zero, so that \[\rk_\Z\SFH(M_T,\gamma_T) = \rk_\Z\SFC(\Heeg_T) = |\gen(\Heeg_T)| = |\gen(\Heeg)|.\] This was the inspiration for our result above.
\end{remark}

We will need the following  for the proof of Proposition \ref{prop:gens}; see \cite[Corollary 4.3]{ghosh-li}.

\begin{proposition}
\label{prop:decompose}
Suppose $(M,\gamma)$ is a balanced sutured manifold and $D\subset M$ is a properly embedded disk which intersects $\gamma$ in four points. Then \[\SHI(M,\gamma) \cong \SHI(M',\gamma')\oplus \SHI(M'',\gamma''),\] where $(M',\gamma')$ and $(M'',\gamma'')$ are the decompositions of $(M,\gamma)$ along $D$ and $-D$, respectively.
\end{proposition}

\begin{proof}
Ghosh and Li prove this in \cite[Corollary 4.3]{ghosh-li} under the assumption that $(M,\gamma)$ is taut and at least one of $(M',\gamma')$ and $(M'',\gamma'')$ is taut. We will show that this additional assumption is unnecessary, by showing that the proposition still holds when the assumption is not true.

First,  suppose $(M,\gamma)$ is not taut. Then neither $(M',\gamma')$ nor $(M'',\gamma'')$ is taut, by \cite[Lemma 0.4]{gabai-foliations2}. In this case, we have \[\SHI(M,\gamma) \cong \SHI(M',\gamma')\cong \SHI(M'',\gamma'') =0,\] and the proposition holds. 

\begin{figure}[ht]
\labellist
\small \hair 2pt
\pinlabel $D$ at 6 71
\pinlabel $(M,\gamma)$ at 150 18
\pinlabel $(M,\gamma_1')$ at 246 18
\pinlabel $(M,\gamma_1'')$ at 340 18
\endlabellist
\centering
\includegraphics[width=10.5cm]{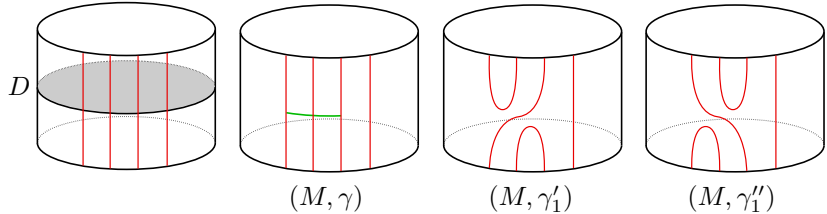}
\caption{Left, a neighborhood of $D\subset M$; the suture $\gamma$ is shown in red. Second to left, the arc of attachment for the initial bypass in the triangle in green.}
\label{fig:bypass}
\end{figure}

Next, suppose neither $(M',\gamma')$ nor $(M'',\gamma'')$ is taut. Consider the bypass exact triangle of \cite{bs-trefoil}: \[ \xymatrix@C=-35pt@R=30pt{
\SHI(-M,-\gamma) \ar[rr] & & \SHI(-M,-\gamma'_1) \ar[dl] \\
& \SHI(-M,-\gamma''_1), \ar[ul] & \\
} \]
determined by an initial bypass attachment to $(M,\gamma)$ along an arc in $\partial D$ as shown in Figure \ref{fig:bypass}. The other manifolds $(M,\gamma'_1)$ and $(M,\gamma''_1)$ in the triangle  product disk decompose (along a copy of $D$ which intersects the new sutures in two points) to $(M',\gamma')$ and $(M'',\gamma'')$, respectively. Therefore, since $\SHI$ is invariant under product disk decomposition, we have \begin{align*}
\SHI(M,\gamma'_1)&\cong \SHI(M',\gamma')\cong 0,\\
\SHI(M,\gamma''_1)&\cong \SHI(M'',\gamma'')\cong 0.
\end{align*} It then follows from the bypass triangle, and the symmetry of $\SHI$ under orientation reversal, that $\SHI(M,\gamma)=0$ as well, so the proposition holds.
\end{proof}

\begin{proof}[Proof of Proposition \ref{prop:gens}]

Recall that $D_{ij}^2$ denotes  a tubular neighborhood of $p_{ij}\in R_i$. There exists a (possibly empty) set of disjoint, properly embedded arcs \[d_1,\dots,d_m\subset \Sigma-\bigcup_{i=1}^{n}\bigcup_{j=1}^{a_i} D_{ij}\] which satisfy the following three conditions:
\begin{enumerate}
\item for every $e=1,\dots,m$, the arc $d_e$ is contained in some region of $\Sigma-\alpha-\beta$, 
\item for every $e$, either both endpoints of $d_e$ are  on $\partial\Sigma$, or each is on some $\partial D_{ij}^2,$ and
\item $\Sigma-\bigcup_{i=1}^{n}\bigcup_{j=1}^{a_i} D_{ij} - d_1-\dots - d_m$ deformation retracts onto $\alpha\cup\beta$.
\end{enumerate} Now consider the disk \[\delta_e = d_e\times[-1,1]\subset \Big(\Sigma-\bigcup_{i=1}^{n}\bigcup_{j=1}^{a_i} D_{ij} \Big)\times [-1,1] \subset \partial M_T,\] for $e=1,\dots,m$. The boundary of each $\delta_e$ intersects  $\gamma_T$ in  two points; hence, $\delta_e$ is a product disk. 
Since $\SHI$ is invariant under product disk decomposition, let $(M_T,\gamma_T)$ henceforth refer to the balanced sutured manifold obtained after   decomposing along  $\delta_1,\dots,\delta_m$.

$M_T$ then admits the following description.  Let $q_1,\dots,q_t$ denote the intersection points between the $\alpha$ and $\beta$ curves. If $q_\ell$ is an intersection point between $\alpha_i$ and $\beta_j$, let $r_\ell\subset \Sigma$ denote the rectangular component of $A_i\cap B_j$ which contains $q_\ell$. Then $M_T$ is the given by the union \[M_T = \D_{\alpha_1}\cup \dots \cup \D_{\alpha_k} \cup \D_{\beta_1}\cup \dots \cup \D_{\beta_k} \cup \tau_1  \cup\dots \cup \tau_t \] of the usual 2-handles with the \emph{tubes} \[\tau_\ell = r_\ell\times[-1,1],\] as shown in Figure \ref{fig:tubes}. Let $c_\ell$ denote the union of the four corners of the rectangle $r_\ell$. Then the suture $\gamma_T$ is given by \[\gamma_T = \bigcup_{i=1}^k (\partial D_{\alpha_i}^2 \times \partial I)\cup \bigcup_{i=1}^k (\partial D_{\beta_i}^2 \times \partial I) \cup \bigcup_{\ell=1}^t (c_\ell \times[-1,1]) - \bigcup_{\ell=1}^t (\partial r_\ell\times\{-1,1\})\] as shown and oriented in the figure near a tube $\tau_\ell$. Let \[m_\ell = r_\ell\times\{0\}\subset \tau_\ell\] denote the meridional disk of  $\tau_\ell$, oriented as in Figure \ref{fig:tubes}, for $\ell=1,\dots,t$. Note that the boundary of each  $m_\ell$ intersects the suture $\gamma_T$ in four points.

\begin{figure}[ht]
\labellist
\small \hair 2pt
\pinlabel $\alpha_i$ at 85 274
\pinlabel $\beta_j$ at 39 234
\pinlabel $\Sigma$ at 85 214
\pinlabel $\D_{\beta_i}$ at 308 290
\pinlabel $\D_{\alpha_i}$ at 308 198
\pinlabel $m_\ell$ at 328 242
\pinlabel $\epsilon_\ell=-$ at 115 30
\pinlabel $\epsilon_\ell=+$ at 310 30
\endlabellist
\centering
\includegraphics[width=13.5cm]{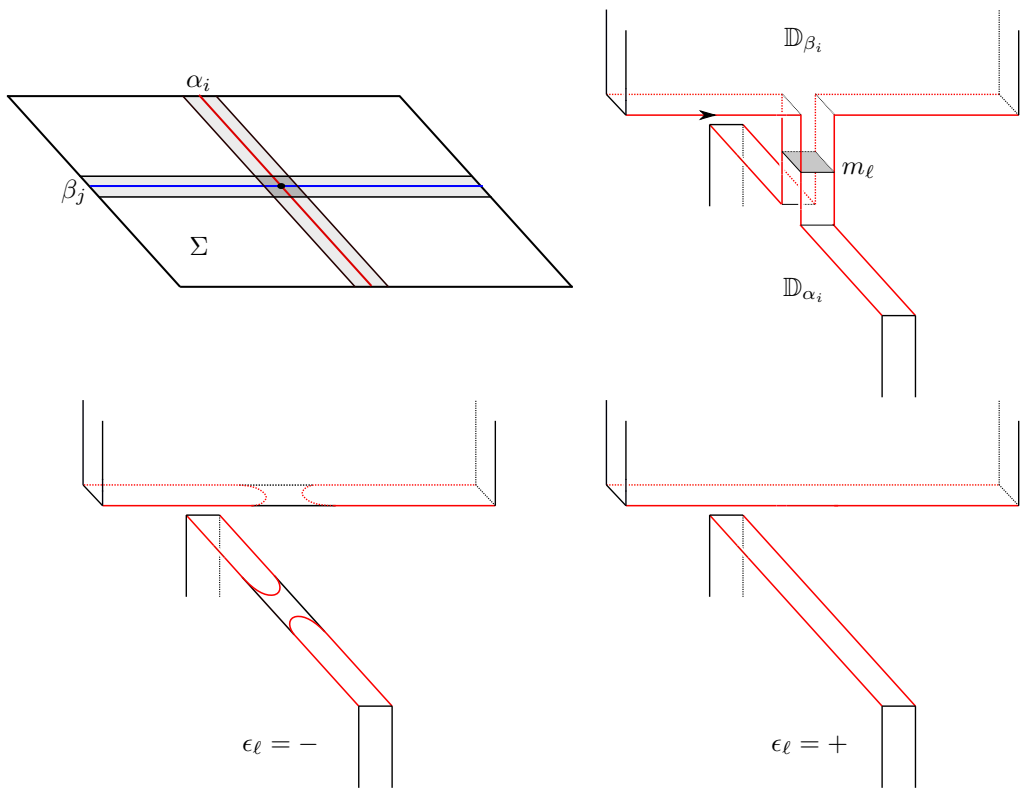}
\caption{Top left, an intersection point $q_\ell\in\alpha_i\cap \beta_j$. The rectangular component $r_\ell$ of $A_i\cap B_j$ containing $q_\ell$ is shown in darker gray. Top right, the 2-handles $\D_{\alpha_i}$ and $\D_{\beta_j}$ glued together by the tube $\tau_\ell$. The suture $\gamma_T$ is shown in red. The meridional disk $m_\ell$ is shown in  gray; its oriented normal points upwards. Bottom left, the result of decomposing along $-m_\ell$. Bottom right, the result of decomposing along $m_\ell$.}
\label{fig:tubes}
\end{figure}

For each $t$-tuple of signs \[I = (\epsilon_1,\dots,\epsilon_t)\in \{+,-\}^t,\] let $(M_T^I,\gamma_T^I)$ be the sutured manifold obtained by decomposing $(M_T,\gamma_T)$ along the disks \[\epsilon_1m_1 \cup \dots \cup  \epsilon_tm_t.\] Then \[\SHI(M_T,\gamma_T)\cong \bigoplus_{I \in \{+,-\}^t} \SHI(M_T^I,\gamma_T^I),\] by Proposition \ref{prop:decompose}. Each $(M_T^I,\gamma_T^I)$ is  simply a union of  3-balls\[M_T^I\cong \D_{\alpha_1}\cup \dots \cup \D_{\alpha_k}\cup\D_{\beta_1}\cup \dots \cup \D_{\beta_k},\]  which means that  $\SHI(M_T^I,\gamma_T^I)$ is either $\C$ or trivial, according as whether $\gamma_T^I$ has exactly one component on the boundary of each of these 3-balls or not. We claim that the nonzero summands $\SHI(M_T^I,\gamma_T^I)$ are in one-to-one correspondence with the elements of $\gen(\Heeg)$, which will then complete the proof. 

For this claim, we consider the restriction of $\gamma_T^I$ to the ball $\D_{\alpha_i}$. Let \[q_{\ell_1},\dots,q_{\ell_p}\in\{q_1,\dots,q_t\}\] denote the intersection points between $\alpha_i$ and $\beta$. Then $\gamma_T^I$ restricts to exactly one component on the boundary of $\D_{\alpha_i}$ iff exactly one of $\epsilon_{\ell_1},\dots \epsilon_{\ell_p}$ is $-$ and the rest are $+$. The analogous statement holds for the restriction of $\gamma_T^I$ to $\D_{\beta_j}$. Thus, if we let \[q(I) = \{q_{\ell}\mid \epsilon_\ell = -\},\] then  $\gamma_T^I$ restricts to exactly one component on each 3-ball in $M_T^I$ iff $q(I)\in \gen(\Heeg)$.
\end{proof}

\begin{remark}
Proposition \ref{prop:gens} also follows from the fact that $\SHI$ and $\SFH$ obey the same decomposition laws (Proposition \ref{prop:decompose}, and the invariance under product disk decomposition), agree in rank for sutured 3-balls, and \[\rk_\Z\SFH(M_T,\gamma_T) = |\gen(\Heeg)|,\] per Remark \ref{rmk:SFH}. Our original  proof    has the advantages that it does not rely on the definition of the differential in $\SFC$, and  it establishes a very concrete bijection between the nonzero summands $\SHI(M_T^I,\gamma_T^I)$ and elements of $\gen(\Heeg)$.
\end{remark}

\subsection{The proof}\label{ssec:admissible}

Recall that a sutured Heegaard diagram $\Heeg$ for a balanced sutured manifold $(M,\gamma)$  is \emph{admissible} iff  every nontrivial periodic domain has both positive and negative multiplicities \cite{juhasz-sutured}.  This is  automatically true of any $\Heeg$ when  $H_1(M,\partial M;\Q)=0$ (there are no nontrivial periodic domains in this case), though every balanced sutured manifold admits an admissible diagram.

\begin{proof}[Proof of Theorem \ref{thm:main}]
Let $\Heeg = (\Sigma,\alpha,\beta)$ be an admissible sutured Heegaard diagram for $(M,\gamma)$. Then we can assign a positive integer area $a_i$ to each region $R_i$ of $\Sigma-\alpha-\beta$ disjoint from $\partial \Sigma$, so that the  signed area of every periodic domain is zero; see  \cite[Lemma 4.2]{osz-props}.\footnote{Ozsv{\'a}th and Szab{\'o} state this for real-valued areas, but proof shows the same is true for integer areas.} 
Fix $a_i$ distinct points $p_{i1},\dots,p_{ia_i}\in R_i$ for each $i$, and let \[T = \bigcup_{i=1}^n\bigcup_{j=1}^{a_i} T_{ij} \subset M\] be the corresponding full tangle for $\Heeg$, as in \S\ref{ssec:full}.

 We claim that  $[T]=0$ in $H_1(M,\partial M;\Q)$. To see this, note that for every periodic domain $P$ of $\Heeg$, the intersection number  of $T$ with the 2-cycle in $M$ represented by $P$ is negative the signed area of $P$, which is zero. Since the homology classes represented by periodic domains span $H_2(M)$, the claim follows. Theorem \ref{thm:inequality2} therefore implies that \[\dim_\C\SHI(M,\gamma)\leq \dim_\C\SHI(M_{T},\gamma_{T}).\] Theorem \ref{thm:main} then follows from the fact that \[ \dim_\C\SHI(M_{T},\gamma_{T})=|\gen(\Heeg)|,\] by Proposition \ref{prop:gens}.
\end{proof}

\begin{remark}\label{rmk:admissible}It is  not true that the inequality \[\dim_\C\SHI(M,\gamma)\leq |\gen(\Heeg)|\] holds for \emph{any}  sutured Heegaard diagram $\Heeg$ for an arbitrary balanced sutured $(M,\gamma)$. For example, consider the diagram $\Heeg = (T^2-D^2,\alpha_1,\beta_1)$ for \[(M,\gamma) = ((S^1\times S^2)(1),\delta)\] in which $\alpha_1$ and $\beta_1$ are disjoint curves on the punctured torus. In this case, we know that \[\dim_\C\SHI(M,\gamma) = 2,\] while $ |\gen(\Heeg)|=0$. The issue here is that  $\Heeg$  is not admissible.
\end{remark}

\section{Further directions}\label{sec:further}
Let $\Heeg$ be an admissible sutured Heegaard diagram for a balanced sutured manifold $(M,\gamma)$. Let $T$ be a full tangle for $\Heeg$, as defined in \S\ref{sec:proof}. In ongoing work, we prove that \[\SHI(-M_T,-\gamma_T)\cong\C^{|\gen(\Heeg)|}\] has  a basis given by the contact invariants  of the tight contact structures on $(M_T,\gamma_T)$. In particular,  this sutured instanton homology group is  naturally graded by homotopy classes of 2-plane fields. We discuss potential applications of this fact below.

Let $T'$ be a mixed tangle for  $T$, as defined in \S\ref{sec:inequality}. Let $V_{T}$ and $V_m$ be the groups    \begin{align*}
V_T&=\SHI(-M_{T'},-\gamma_{T'},(A_2^-,\dots,A_n^-), (0,\dots,0)),\\
V_m&=\SHI(-M_{T'},-\Gamma_m, (A_2,\dots,A_n), (0,\dots,0))
\end{align*}  from \S\ref{sec:inequality}, for $m\in \mathbb{N}$. To prove Theorem \ref{thm:main}, we proved that $V_T\cong \SHI(-M_T,-\gamma_T)$ and  \[\dim_\C \SHI(-M,-\gamma) = \dim_\C V_{m+1} - \dim_\C V_m\leq \dim_\C V_T\] for sufficiently large $m$. But, in fact, this inequality can be viewed as coming from a spectral sequence similar to that in \cite[Section 4]{li-ye}. This spectral sequence can also be described as follows. From the bypass exact triangle used in \S\ref{sec:inequality}, we have that \[V_T\cong H_*(\Cone(\psi_-:V_m\to V_{m+1})),\] where $\psi_-$ is the map associated to the bypass attachment along  the arc $\eta_-$. One can prove, on the other hand, that \[\SHI(-M,-\gamma)\cong H_*(\Cone(\psi_--\psi_+:V_m\to V_{m+1}))\] for $m$ large, where $\psi_+$ is a related bypass attachment map. The groups $V_m$ and $V_{m+1}$ can be graded using the rational Seifert surface for $T'$, as in \cite{li}. After adjusting this grading by an overall shift, the map $\psi_-$ is grading-preserving while $\psi_+$ decreases the grading by $1$, for $m$ large. The complex $\Cone(\psi_--\psi_+)$ is then filtered, and the $E_1$ page of the associated spectral sequence is $H_*(\Cone(\psi_-))$. In sum, we have a spectral sequence \begin{equation}\label{eqn:ss}\C^{|\gen(\Heeg)|}\cong V_T\cong H_*(\Cone(\psi_-)) \implies H_*(\Cone(\psi_--\psi_+))\cong \SHI(-M,-\gamma).\end{equation}

The first potential application of these ideas involves defining a grading on $\SHI(-M,-\gamma)$ by homotopy classes of plane fields. Indeed, $\SHI(-M_T,-\gamma_T)$ has such a grading, as mentioned above, as it is generated by contact invariants of contact structures. The manifold $(M_{T'},\gamma_{T'})$ is obtained by gluing $(M_T,\gamma_T)$ along annuli, as in \S\ref{sec:inequality}, and we believe that the tight contact structures on the latter glue to give tight contact structures on the former whose invariants form a basis for $V_T$. So, there should be a natural grading by homotopy classes of 2-plane fields on $V_T$ as well. The bypass maps $\psi_-,\psi_+$ are natural from a contact-geometric standpoint, and should therefore shift  plane field gradings in a sensible way. We expect that one can then use the relation between $V_T$ and $\Cone(\psi_-)$ and the structure of the latter to define a plane field grading on $\Cone(\psi_-)$, and then on $\Cone(\psi_--\psi_+)$.

A  grading by homotopy classes of 2-plane fields on $\SHI$ would  enable one to define $\textrm{Spin}^c$ decompositions of these groups, as well as an analogue of the Maslov grading in Heegaard Floer homology (see \cite[Section 4]{li-ye} for another approach to such a decomposition). The current lack of such structure makes it difficult to translate arguments from the Heegaard Floer setting to the instanton Floer setting.

A related second application is towards proving the isomorphism \eqref{eqn:iso1}.
Indeed, there is some hope that one could understand the spectral sequence \eqref{eqn:ss} purely in terms of contact geometry, and thereby obtain a more axiomatic proof that  \begin{equation}\label{eqn:isoall}\SHI(M,\gamma) \cong \SFH(M,\gamma)\otimes \C\cong \SHM(M,\gamma)\otimes \C,\end{equation} since the analogous spectral sequences can be defined in the Heegaard Floer and monopole Floer settings by the same contact-geometric means. 

A more pedestrian approach to \eqref{eqn:isoall} is the following: first, prove that one can understand the spectral sequence \[\C^{|\gen(\Heeg)|}\implies \SHI(-M,-\gamma)\] as coming from a differential \[\partial=\partial_1 + \partial_2 + \dots:\C^{|\gen(\Heeg)|}\to \C^{|\gen(\Heeg)|},\] where $\partial_k$ shifts a grading (coming from homotopy classes of 2-plane fields) on $\C^{|\gen(\Heeg)|}\cong V_T$ by $k$, such that \[\SHI(-M,-\gamma)\cong H_*(\C^{|\gen(\Heeg)|},\partial).\] Then, for generators $x,y\in\gen(\Heeg)$ and the corresponding basis elements $e_x,e_y\in\C^{|\gen(\Heeg)|}$, perhaps one could  use the 2-plane field gradings to show that the coefficient \begin{equation}\label{eqn:coeff}\langle \partial e_x,e_y\rangle \end{equation} is nonzero only if there is a homotopy class of Whitney disks \begin{equation}\label{eqn:class}\varphi\in\pi_2(x,y)\end{equation} with positive domain in $\Heeg$ and Maslov index one. If  even this were true, then one could prove, for example, that the inequality in Corollary \ref{cor:simple} is an equality, \[\dim_\C\KHI(L(p,q),K)\leq \rk_\Z\HFK(L(p,q),K)\] for  $(1,1)$-knots $K\subset L(p,q)$. More generally, the hope would be that for a \emph{nice} diagram $\Heeg$ (the regions of $\Sigma-\alpha-\beta$ disjoint from $\partial\Sigma$ are  bigons or rectangles), one could show that the coefficient \eqref{eqn:coeff} is $\pm 1$ iff there is a  class  as in \eqref{eqn:class} with positive domain and Maslov index one (the domain of such a class is necessarily an embedded bigon or rectangle in this case). This would be enough to prove \eqref{eqn:iso1} by \cite{sw}, and then \eqref{eqn:isoall} by the same methods.

\bibliographystyle{alpha}
\bibliography{References}

\begin{thebibliography}{ABDS20}

\bibitem[ABDS20]{abds}
Antonio Alfieri, John~A. Baldwin, Irving Dai, and Steven Sivek.
\newblock Instanton {F}loer homology of almost-rational plumbings.
\newblock arXiv:2010.03800, 2020.

\bibitem[BS16]{bs-instanton}
John~A. Baldwin and Steven Sivek.
\newblock Instanton {F}loer homology and contact structures.
\newblock {\em Selecta Math. (N.S.)}, 22(2):939--978, 2016.

\bibitem[BS18]{bs-trefoil}
John~A. Baldwin and Steven Sivek.
\newblock Khovanov homology detects the trefoils.
\newblock arXiv:1801.07634, 2018.

\bibitem[BS20]{BSconcordance}
John~A. Baldwin and Steven Sivek.
\newblock Framed instanton homology and concordance.
\newblock arXiv:2004.08699, 2020.

\bibitem[Gab87]{gabai-foliations2}
David Gabai.
\newblock Foliations and the topology of {$3$}-manifolds. {II}.
\newblock {\em J. Differential Geom.}, 26(3):461--478, 1987.

\bibitem[GL19]{ghosh-li}
Sudipta Ghosh and Zhenkun Li.
\newblock Decomposing sutured monopole and instanton {F}loer homologies.
\newblock arXiv:1910.10842, 2019.

\bibitem[Juh06]{juhasz-sutured}
Andr\'as Juh\'asz.
\newblock Holomorphic discs and sutured manifolds.
\newblock {\em Algebr. Geom. Topol.}, 6:1429--1457, 2006.

\bibitem[KM10]{km-excision}
Peter Kronheimer and Tomasz Mrowka.
\newblock Knots, sutures, and excision.
\newblock {\em J. Differential Geom.}, 84(2):301--364, 2010.

\bibitem[Lek13]{lekili2}
Yank{\i} Lekili.
\newblock Heegaard-{F}loer homology of broken fibrations over the circle.
\newblock {\em Adv. Math.}, 244:268--302, 2013.

\bibitem[Li19]{li}
Zhenkun Li.
\newblock Knot homologies in monopole and instanton theories via sutures.
\newblock arXiv:1901.06679, 2019.

\bibitem[LL12]{levine-lewallen}
A.S. Levine and Sam Lewallen.
\newblock {Strong L-spaces and left orderability}.
\newblock {\em Math. Res. Lett.}, 19(6):1237--1244, 2012.

\bibitem[LPCS20]{lpcs}
Tye Lidman, Juanita Pinz{\'o}n-Caicedo, and Christopher Scaduto.
\newblock Framed instanton homology of surgeries on {L}-space knots.
\newblock arXiv:2003.03329, 2020.

\bibitem[LY]{li-ye2}
Zhenkun Li and Fan Ye.
\newblock {Floer homology and Euler characteristics}.
\newblock in preparation.

\bibitem[LY20]{li-ye}
Zhenkun Li and Fan Ye.
\newblock Instanton floer homology, sutures, and heegaard diagrams.
\newblock arXiv:2010.07836, 2020.

\bibitem[OS04]{osz-props}
Peter Ozsv\'{a}th and Zolt\'{a}n Szab\'{o}.
\newblock {Holomorphic disks and three-manifold invariants: Properties and
  applications}.
\newblock {\em Ann. Math.}, 159:1159--1245, 2004.

\bibitem[SW10]{sw}
Sucharit Sarkar and Jiajun Wang.
\newblock {An algorithm for computing some Heegaard Floer homologies}.
\newblock {\em Ann. Math.}, 171(2):1213--1236, 2010.

\bibitem[Wan20]{wang}
Joshua Wang.
\newblock The cosmetic crossing conjecture for split links.
\newblock arXiv:2006.01070, 2020.

\end{thebibliography}

\end{document}